\documentclass[reqno,twoside]{amsart}
\usepackage{amsfonts}
\usepackage{amsmath,amssymb}
\usepackage{epic}
\usepackage{pstricks}
\usepackage{graphicx}
\usepackage{xcolor}

\usepackage{a4wide,amsmath,amssymb,latexsym,amsthm}
\usepackage{eucal}
\usepackage{graphicx}

\newtheorem{theorem}{Theorem}[section]
\newtheorem{proposition}[theorem]{Proposition}
\newtheorem{remark}[theorem]{Remark}
\newtheorem{lemma}[theorem]{Lemma}

\numberwithin{equation}{section}

\graphicspath{{./Imgs/}}

\newcommand{\R}{\mathbb R}
 
\newcommand{\N}{\mathbb N}

\newcommand{\A}{\mathcal A}
\newcommand{\D}{\displaystyle}
\newcommand{\be}{\begin{equation}}
\newcommand{\ee}{\end{equation}}
\newcommand{\ba}{\begin{eqnarray}}
\newcommand{\ea}{\end{eqnarray}}
\newcommand{\beq}{\begin{equation}}
\newcommand{\eeq}{\end{equation}}

\usepackage{color}
\definecolor{red}{rgb}{0,0,0}

\usepackage[textsize=small]{todonotes}

\numberwithin{equation}{section}

\def\N{{\mathbb{N}}}

\def\A{{\bf A}}

\usepackage{color}
\usepackage{stmaryrd}

\keywords{Carleman inequalities, Bukhgeim--Klibanov method, hidden regularity, Moore--Gibson--Thompson equation}

\begin{document}

\title[Inverse problem MGT]{An inverse problem for Moore--Gibson--Thompson equation arising in high intensity ultrasound}

\author{R. Arancibia}
\address{R. Arancibia, Universidad T\'ecnica Federico Santa Mar\'ia, Departamento de Matem\'atica, Casilla 110-V, Valpara\'iso, Chile}
\email{rogelio.arancibia.13@sansano.usm.cl}

\author{R. Lecaros}
\address{R. Lecaros, Universidad T\'ecnica Federico Santa Mar\'ia, Departamento de Matem\'atica, Casilla 110-V, Valpara\'iso, Chile}
\email{rodrigo.lecaros@usm.cl}

\author{A. Mercado}
\address{A. Mercado, Universidad T\'ecnica Federico Santa Mar\'ia, Departamento de Matem\'atica, Casilla 110-V, Valpara\'iso, Chile}
\email{alberto.mercado@usm.cl}

\author{S. Zamorano}
\address{S. Zamorano, University of Santiago of Chile (USACH), Faculty of Science, Mathematics and Computer Science Department, Casilla 307, Correo 2, Santiago, Chile.}
 \email{sebastian.zamorano@usach.cl}

\thanks{R. Arancibia was partially supported by the Direcci\'on de Postgrados y Programas (DPP) of the U. T\'ecnica Federico Santa Mar\'ia. R. Lecaros was partially supported by FONDECYT(Chile) Grant NO: 11180874. The work of A. Mercado  was partially supported by FONDECYT (Chile) Grant NO: 1171712. The second and third author was partially supported by BASAL Project, CMM - U. de Chile. S. Zamorano was supported by the FONDECYT(Chile) Postdoctoral Grant NO: 3180322 nd by ANID PAI Convocatoria Nacional Subvenci\'on a la Instalaci\'on en la Academia Convocatoria 2019 PAI 77190106.}

\begin{abstract}
In this article we study the inverse problem  of recovering  a space-dependent coefficient of the Moore--Gibson--Thompson (MGT) equation, from  knowledge of  the  trace of the solution on
some  open subset of  the boundary.  We obtain  the  Lipschitz stability   for this inverse problem, and we design a  convergent  algorithm  for the reconstruction of  the unknown coefficient. The techniques  used  are based on  Carleman inequalities for wave equations and properties of the MGT equation.
\end{abstract}

\maketitle


\section{Introduction}

Let $\Omega\subseteq\R^{N}$ be a nonempty bounded open set (for $N=2$ or $N=3$), with  a smooth  boundary $\Gamma$, and let $T>0$. We consider the MGT equation   
\begin{align}\label{eqmain2}
\left\{
\begin{array}{ll}
 \tau u_{ttt}+ \alpha  u_{tt}-c^2\Delta u -b\Delta u_{t}=f,  &\Omega\times(0,T)\\
u=0, & \Gamma \times (0,T) \\
u(\cdot,0) = u_0, \medspace u_t(\cdot,0) = u_1, \medspace u_{tt}(\cdot,0) = u_2,   &\Omega,
\end{array}
\right.
\end{align}
where $\alpha \in L^\infty(\Omega)$, $c \in \R$ and $\tau,b>0$. 

This equation arises  as a linearization of a model for wave propagation in viscous thermally relaxing fluids. In that cases, the space-dependent coefficient $\alpha$ depends on a viscosity of the fluid \cite{kaltenbacher2012exponential}. This third order in time equation
has been studied by several authors from various  points of view. We can mentioned, among others, the works \cite{ kaltenbacher2015mathematics, kaltenbacher2011wellposedness, kaltenbacher2012well,  liu2014inverse, marchand2012abstract, pellicer2019optimal, lizama2019controllability} for a variety of problems related to this equation.

In particular, one interesting characteristic of this equation  is that the structural damping $b$ plays  a crucial role for the well-posedness,  contrary of second order equations with damping ($\tau=0$ and $\alpha>0$ in \eqref{eqmain2}). For instance, in  \cite{kaltenbacher2011wellposedness} it is
proved that, if $b=0$ and $\alpha$ a positive constant, there does not exist an infinitesimal generator  of a semigroup, in contrast with second order equations, where the structural damping does not affect the well-posedness  of the equation. 
The parameter $\gamma:=\alpha-\frac{\tau c^2}{b}$  gives relevant information regarding the stability of the system. If $\gamma>0$, the group associated to the equation is exponentially stable, and for $\gamma=0$, the group is conservative, see for instance \cite{marchand2012abstract}. 
On the other hand, Conejero, Lizama and Rodenas \cite{conejero2015chaotic} proved that the one-dimensional equation exhibits a chaotic behavior
if $\gamma<0$. 
Also, for the case in which  $\alpha$ is given by a function  depending on space and time, 
 the well posedness and the exponential decay was proved by Kaltenbacher and  Lasiecka in 
\cite{kaltenbacher2012exponential}.

 Concerning the well posedness of the system \eqref{eqmain2}, it is known (see \cite[Theorem 2.2]{kaltenbacher2012exponential}) that, 
given a coefficient   $\alpha\in L^{\infty}(\Omega)$ and data satisfying 
\begin{align}\label{dat}
 (u_0,u_1,u_2)\in H^2(\Omega)\cap H_0^1(\Omega)\times H_0^1(\Omega)\times L^2(\Omega),\; 
 f\in L^2(0,T;L^2(\Omega)), 
 \end{align}
  the system \eqref{eqmain2} admits a unique weak solution $(u,u_t,u_{tt})$  satisfying 
$$
(u,u_{t},u_{tt})\in C([0,T]; H^2(\Omega)\cap H_0^1(\Omega)\times H_0^1(\Omega)\times L^2(\Omega)).
$$

 In this article, we study  the inverse problem of recovering the unknown space-dependent coefficient 
$\alpha = \alpha(x)$, the frictional damping term, 
from partial knowledge of some trace of the solution $u$ at the boundary, namely,
\begin{align*}
\D\frac{\partial u}{\partial n}\mbox{ on }\Gamma_0\times(0,T),
\end{align*}
where $\Gamma_0 \subset \Gamma$ is a relatively open subset of the boundary, called the observation region, and $n$ is the outward unit normal vector on $\Gamma$. 
We will often write  $u(\alpha)$  
to denote the dependence of $u$ on the coefficient $\alpha$.

More precisely,  in this paper we  study the following properties  of the stated inverse problem:

\begin{itemize}
\item {\bf Uniqueness:} 
\begin{align*}
\frac{\partial u(\alpha_1)}{\partial n} = \frac{\partial u(\alpha_2)}{\partial n}
\mbox{ on }\Gamma_0\times(0,T)\mbox{ implies }\alpha_1=\alpha_2\mbox{ in }\Omega.
\end{align*}

\item {\bf Stability:} 
\begin{align*}
\|\alpha_1-\alpha_2\|_{X(\Omega)}\leq C\left\|\frac{\partial u(\alpha_1)}{\partial n} - \frac{\partial u(\alpha_2)}{\partial n}\right\|_{Y(\Gamma_0)},
\end{align*}
for some appropriate spaces $X(\Omega)$ and $Y(\Gamma_0)$.

\item {\bf Reconstruction:} 
Design an algorithm to recover the coefficient $\alpha$ from the knowledge of $\D \frac{\partial u(\alpha)}{\partial n}$ on $\Gamma_0$.
\end{itemize}

The first part of this work is concerned with the uniqueness and stability issues of the inverse problem. 
We obtain a  stability result, which directly  implies a uniqueness one,  under certain conditions for $\alpha$, $\Gamma_0$ and the time $T$. 
We use the Bukheim--Klibanov method, which is based on 
 the so-called Carleman estimates. 
 We  prove a Carleman  estimate for MGT equation, which will be based on the Carleman inequality for wave operator given in \cite{baudouin2013global}.

The second part of this work is focused on giving a constructive and iterative algorithm which allows us to find the coefficient $\alpha$ from the knowledge of the additional data $\frac{\partial u}{\partial n}$ on the observation domain $\Gamma_0$. For that, we 
study an appropriate  functional,  and we  show that this functional admits a unique minimizer on a suitable  domain.
Using this results, we  will prove the convergence of an  iterative algorithm. We refer to Section \ref{s4} for details.
This algorithm is adapted from   \cite{baudouin2013global}, 
where it was introduced an algorithm for recovering zero-order terms in  the wave equation.
We can also mention the works of Beilina and Klibanov \cite{beilina2012approximate, beilina2015globally}, 
%
%
where the authors studied the reconstruction of a coefficient in a hyperbolic equation using the Carleman weight.

The remaining  of this paper is organized as follows. In Section \ref{s2-1} we present our main results:
 Theorem \ref{mainresult},  which establishes  the stabilization property of our inverse problem and a Carleman type estimate which is contained in Theorem \ref{observability}.
In section \ref{s2} we present some auxiliary results of the MGT equation which are needed for the inverse problem. Besides, we prove the hidden regularity for the MGT equation. In section \ref{s3} we prove the main results  of our work, namely Theorems \ref{mainresult} and \ref{observability}. Finally, in section \ref{s4} we focus on the algorithm for the reconstruction of coefficient $\alpha$ and we prove the convergence of this Algorithm.

\section{Statement of the main results}\label{s2-1}
In this section we state our main results concerning the inverse problem proposed in the Introduction.  In order to state the precise result that we obtain, we consider the following set of  admissible coefficients:
\begin{align}\label{asuma}
{\mathcal A}_M  = \left\{  \alpha \in L^\infty(\Omega),   \quad  \frac{c^2}{b} \leq \alpha(x) \leq M  \quad  \forall x\in\overline{\Omega} \right\},
\end{align}
and  the  geometrical assumptions, sometimes referred to as the Gamma--condition of Lions or the multiplier condition:
\begin{align}\label{asumb}
\exists x_0\notin\Omega\mbox{ such that }\Gamma_0\supset\{x\in\Gamma: \ (x-x_0)\cdot n\geq 0\},
\end{align}
 and 
\begin{align}\label{asumc}
T>\sup_{x\in\Omega}|x-x_0|.
\end{align}

Henceforth we will set $\tau=1$ for simplicity. 
Our main result concerns  the  stability of the inverse problem:
\begin{theorem}\label{mainresult}
For  $\Gamma_0 \subset \Gamma$,  $M>0$ and $T >0$ satisfying  \eqref{asumb}-\eqref{asumc}, 
suppose  there exists $\eta >0$ such that 
\begin{align}\label{mineta}
|u_2| \geq \eta > 0 \quad  a.e. \, in \, \, \Omega,
\end{align}
 and $\alpha_2 \in \A_M$ is such that  the unique solution $u(\alpha_2)$ of \eqref{eqmain2} satisfies
\begin{align}\label{regsol}
u(\alpha_2) \in H^3(0,T;L^{\infty}(\Omega)). 
\end{align}
Then there exists a constant $C>0$ such that
\begin{align}\label{stabi}
C^{-1}\|\alpha_1-\alpha_2\|_{L^2(\Omega)}^2 
\leq
 \left\| \frac{\partial u(\alpha_1)}{\partial n} -  \frac{\partial u(\alpha_2)}{\partial n} \right\|_{H^2(0,T;L^2(\Gamma_0))}^2
\leq  C\|\alpha_1-\alpha_2\|_{L^2(\Omega)}^2
\end{align}
for all $\alpha_1 \in \A_M$.
\end{theorem}

Let us mention some comments about Theorem \ref{mainresult}.

\begin{remark} \label{reg}
The hypothesis $u(\alpha_2)\in H^3(0,T;L^{\infty}(\Omega))$ in Theorem \ref{mainresult} is satisfied if  more regularity is imposed  on the data.
For instance,
taking 
  $m>\frac{N}{2}+1$, 
 it is enough to take  $(u_0,u_1,u_2)\in (H^{m+2}(\Omega)\times H^{m+1}(\Omega)\times H^m(\Omega))$, $\alpha_2\in H^{m-1}(\Omega)$, $f\equiv 0$ and  appropriate boundary compatibility conditions.
Indeed, by Theorem 2.2 in \cite{kaltenbacher2012exponential}, we obtain $$u=u(\alpha_2)\in C([0,T]; H^{m+2}(\Omega))\cap C^1([0,T]; H^{m+1}(\Omega))\cap C^2([0,T]; H^{m}(\Omega)).$$ 
Then, from equation  \eqref{eqmain2} and taking into account that, for $s>\frac{N}{2}$, the Sobolev space $H^{s}(\Omega)$  is an algebra, we have that
$(u_{tt}(\cdot,0), u_{ttt}(\cdot,0), u_{tttt}(\cdot,0))\in (H^{m}(\Omega)\times H^{m-1}(\Omega)\times H^{m-2}(\Omega))$.
 %
Therefore, using again Theorem 2.2 in \cite{kaltenbacher2012exponential}, we deduce that  $u_{tt}\in C([0,T]; H^{m}(\Omega))\cap C^1([0,T]; H^{m-1}(\Omega))\cap C^2([0,T]; H^{m-2}(\Omega))$. 
 Hence 
\begin{align*}
u_{ttt}\in C([0,T];H^{m-1}(\Omega))\cap C^1([0,T];H^{m-2}(\Omega)),
\end{align*}
and using Sobolev's embedding theorem, we get that  $u_{ttt} \in L^2(0,T;L^{\infty}(\Omega))$. 
\end{remark}

\begin{remark}
The inverse problem studied in this paper was previously considered by Liu and Triggiani \cite[Theorem 15.5]{liu2014inverse}. 
They considered $\alpha\in H^{m}(\Omega)$ 
and
initial data $(u_0,u_1,u_2)\in (H^{m+2}(\Omega)\times H^{m+1}(\Omega)\times H^m(\Omega))$ 
with  $m>\frac{N}{2}+2$. By  using  Carleman estimates for a general hyperbolic equation, the authors proved  global uniqueness of any damping coefficient $\alpha$ with boundary measurement given by
\begin{align*}
\frac{c^2}{b}\frac{\partial u}{\partial n}+\frac{\partial u_t}{\partial n},\quad \mbox{on }\Gamma_0\times [0,T],
\end{align*}

and  the initial data is supposed to satisfy  \eqref{mineta} and 
\begin{align}\label{extracondition}
\frac{c^2}{b}u_0(x)+u_1(x)=0, \quad x\in\Omega.
\end{align}
In this paper, using an appropriate  Carleman inequality and the method of Bukhgeim--Klibanov, 
we obtain stability around any regular state, under hypothesis  $m>\frac{N}{2}+1$  and without the additional assumption \eqref{extracondition}.
\end{remark}

\begin{remark}
The hypotheses \eqref{asumb} and \eqref{asumc} on $\Gamma_0$ and $T$  typically arises 
in the study of stability or observability inequalities 
for the wave equation, 
see \cite{ho} where  the multiplier method is used, or 
 \cite{fursikov1996controllability, zhang2000explicit} 
 where  some  observability inequalities are obtained from Carleman estimates.
These hypotheses provide a particular case of the geometric control condition stated in \cite{bardos1992sharp}.
\end{remark}

\begin{remark}
The  assumption of the  positivity for $|u_2|$ appearing in Theorem \ref{mainresult} is classical 
when applying the Bukhgeim-Klibanov method 
and Carleman estimates for inverse problems with only one boundary measurement, see  \cite{baudouin2010lipschitz, liu2011global, yamamoto1999uniqueness}.
\end{remark}

As we mentioned before, in order to study   the stated inverse problem, we use  global Carleman estimates and the method of Bukhgeim--Klibanov, introduced in  \cite{BK}. To state our Carleman estimates precisely, we shall need the following notations.

Assume that $\Gamma_0$ satisfies \eqref{asumb} for some $x_0\in \R^{N}\setminus\overline{\Omega}$. For $\lambda>0$, we define the weight function 
\begin{align}\label{4-SZ}
\varphi_{\lambda}(x,t)=e^{\lambda\phi(x,t)}, \quad (x,t)\in\Omega\times (-T,T),
\end{align}
 where 
\begin{align}\label{5-SZ}
\phi(x,t)=|x-x_0|^2-\beta t^2+M_0 
\end{align}
for some $\beta \in (0,1)$ to be chosen later, and  for some $M_0$  such that 
$ \phi \geq1$, for example any constant satisfying $M_0 \geq \beta T^2 +1$.
To prove  Theorem \ref{mainresult}, we shall use the following Carleman estimate.

\begin{theorem}\label{observability}
Suppose that $\Gamma_0$ and $T$ satisfies \eqref{asumb}, \eqref{asumc}. Let $M>0$ and $\alpha\in \mathcal{A}_{M}$. Let $\beta\in (0,1)$ such that
\begin{align}\label{TC22}
\beta T>\sup_{x\in\Omega} |x-x_0 |.
\end{align}
Then, there exists $s_0>0$, $\lambda>0$ and a positive constant $C$ such that 
\begin{multline}\label{1.133}
\sqrt{s}\int_{\Omega}e^{2s\varphi_{\lambda}(\cdot,0)}|y_{tt}(\cdot,0)|^2dx+s\lambda c^4\int_0^{T}\int_{\Omega}e^{2s\varphi_{\lambda}}\varphi_{\lambda}(|y_{t}|^2+|\nabla y|^2)dxdt\\ +s^3\lambda^3c^4\int_0^{T}\int_{\Omega}e^{2s\varphi_{\lambda}}\varphi_{\lambda}^3|y|^2dxdt  +    s\lambda \int_0^{T}\int_{\Omega}e^{2s\varphi_{\lambda}}\varphi_{\lambda}(|y_{tt}|^2+|\nabla y_{t}|^2)dxdt\\ +s^3\lambda^3\int_0^{T}\int_{\Omega}e^{2s\varphi_{\lambda}}\varphi_{\lambda}^3|y_{t}|^2dxdt
\leq 
C  \int_{0}^{T}\int_{\Omega}e^{2s\varphi_{\lambda} }|Ly|^2dxdt
\\+Cs\lambda\int_{0}^{T}\int_{\Gamma_0}e^{2s\varphi_{\lambda}}\left(|\nabla y_{t}\cdot n|^2+c^4|\nabla y\cdot n|^2\right)d\sigma dt,
\end{multline}
for all $s\geq s_0$ 
and 
for all $y\in L^2(0,T;H_0^1(\Omega))$ satisfying $Ly:=y_{ttt}+ \alpha  y_{tt}-c^2\Delta y -b\Delta y_{t} \in L^2(\Omega\times(0,T))$,  $y(\cdot,0)=y_{t}(\cdot,0)=0$ in $\Omega$, and $y_{tt}(\cdot,0)\in L^2(\Omega)$.
\end{theorem}


Let us mention that, in order to obtain estimate 
\eqref{1.133}, we do not follow the classical procedure  of 
decomposing the differential  operator $Ly$ of  the MGT system. 
Instead of  that, we 
use  in an appropriate way the well-known Carleman estimate for the wave operator, 
from which we are able to obtain \eqref{1.133} thanks to the fact that we are asking  that the initial conditions  $y(\cdot, 0)$ and $y_t(\cdot, 0)$ are null.
For instance, this result is not enough to obtain controllability, but this is coherent with the fact 
 the MGT equation has poor control properties: 
 in \cite{lizama2019controllability}
 is proved that the interior  null controllability of this system is not true, 
  and then,  the boundary  null controllability is also false. 
A similar idea  was considered in \cite{yamamoto2003one},
where a Carleman estimate for the Laplace operator 
 was used to prove the unique continuation property for a linearized Benjamin--Bona--Mahony equation. 

The  Bukhgeim--Klibanov method and Carleman estimates  have been widely used for  obtaining stability of coefficients with  one-measurement observations. 
Concerning inverse problems for  wave  equations  with boundary observations, in \cite{PuelYama}  is studied the problem of recovering a source term of the equation,   \cite{Yama99} 
deals with the problem of recovering a coefficient of the zero-order term, and  \cite{Bella04}  concerns   the recovering of the main coefficient. 
 In addition, we can mention the works \cite{MR1964256, MR3774702} related to coefficient inverse problems for hyperbolic equations.
We refer to  \cite{BellaYama} for an  account of classic and recent results concerning the use of Carleman estimates 
on the  study of inverse problems for hyperbolic equations.

\section{Auxiliary results}\label{s2}
In this section, we state and prove some  auxiliary  results concerning  estimates for the Laplacian of a solution of \eqref{eqmain2}  and a hidden regularity estimate 
for the solution of the MGT equation.

\subsection{Bound of Laplacian of the solutions}
From now, throughout the article, we  define  
\begin{align} \label{gamma}
\gamma(x):=\alpha(x)-\frac{c^2}{b}.
\end{align}
 Let us note that $\alpha\in \mathcal{A}_{M}$ if and only if 
$0 \leq \gamma \leq M$ in 
$\overline{\Omega}$.
%
We also define the energy
\begin{align}\label{1.3}
E_e(y):=&\frac{b}{2}\|\nabla y\|_{L^2(\Omega)}^2+\frac{1}{2}\|y_t\|_{L^2(\Omega)}^2.
\end{align}
In order to prove our main results, some technical estimations are necessary. One of them is the following:
\begin{lemma}\label{Energy}
Let $b>0$ and $M>0$ such that $\alpha\in\mathcal{A}_{M}$. 
Then  there exists $C>0$ such that the total energy 
\begin{align} \label{defener}
\overline{E}(t):=E_e(u_t(t))+E_e(u(t)),
\end{align} 
 satisfies
\begin{align*}
\overline{E}(t)\leq C\left(\overline{E}(0)+\|f\|^2_{L^2(0,T:L^2(\Omega))}\right),\quad t\in [0,T],
\end{align*}
 for every $(u_0,u_1,u_2)\in (H^2(\Omega)\cap H_0^1(\Omega))\times H_0^1(\Omega)\times L^2(\Omega)$ and $f\in L^2(0,T;L^2(\Omega))$, where  $u$ be the unique solution of \eqref{eqmain2}. 
\end{lemma}
\begin{proof}
Without loss of generality, we assume that $b=1$. Then, the equation
$$u_{ttt}+\alpha(x) u_{tt}-c^2 \Delta u-b\Delta u_t=f$$
can be write as follows (recall the definition of $\gamma$ in \eqref{gamma})
\begin{align}\label{1.6}
Lu:=L_0u_t+c^2L_0u+\gamma(x)u_{tt}=f,
\end{align}
where $L_0$ is the wave operator given by 
\begin{align}\label{1.7}
L_0:=\partial_{t}^2-\Delta.
\end{align}

Let us multiply the equation \eqref{1.6} by $u_{tt}(t)+c^2 u_t(t) \in L^2(\Omega)$ and after integrating on $\Omega$, we deduce that

\begin{align*}
\frac{d}{dt}E_e(u_t+c^2u)+\int_\Omega\gamma u_{tt}(u_{tt}+c^2u_t)=\int_\Omega f(u_{tt}+c^2u_t),
\end{align*}
thus, we have 
\begin{align*}
\frac{d}{dt}E_e(u_t+c^2u)+\frac{c^2}{2}\frac{d}{dt}\|\gamma^{1/2}u_{t}\|^2_{L^2(\Omega)} \leq \frac{1}{2}\|f\|^2_{L^2(\Omega)}+E_e(u_t+c^2 u). 
\end{align*}
And using Gronwall's inequality, there exists a constant $C>0$, such that 
\begin{align}\label{eqEner01}
E_e(u_t+c^2u)+\frac{c^2}{2}\|\gamma^{1/2}u_{t}\|^2_{L^2(\Omega)} \leq C (\|f\|^2_{L^2(0,T;L^2(\Omega))}+\overline{E}(0)),\;\forall t\in [0,T].
\end{align}
On other side, a direct computation give us 
\begin{align}\label{1-SZ}
E_e(u_t+c^2u)=E_e(u_t)+c^4E_e(u)+c^2\frac{d}{dt}E_e(u),
\end{align}
and replacing \eqref{1-SZ} in \eqref{eqEner01}, we have 
  \begin{align*}
c^2\frac{d}{dt}E_e(u)\leq C (\|f\|^2_{L^2(0,T;L^2(\Omega))}+\overline{E}(0)).
\end{align*}
Hence, integrating we obtain that, there exists a constant $C>0$, such that 
\begin{align}\label{eqEner02}
E_e(u)\leq C (\|f\|^2_{L^2(0,T;L^2(\Omega))}+\overline{E}(0)),\;\;\;\forall t\in [0,T].
\end{align}
Finally, if we take $\varepsilon<1$, we observe that 
\begin{align}\label{2-SZ}
c^2\frac{d}{dt}E_e(u)=c^2\int_\Omega (u_{tt}u_t+\nabla u\cdot\nabla u_t)\geq -\varepsilon^2 E_e(u_t)-\frac{c^4}{\varepsilon^2}E_e(u),
\end{align}
replacing \eqref{2-SZ} in \eqref{eqEner01} and using \eqref{eqEner02}, we obtain that, there exists a constant $C>0$, such that, 
\begin{align}\label{eqEnerFinal}
E_e(u_t) \leq C (\|f\|^2_{L^2(0,T;L^2(\Omega))}+\overline{E}(0)),\;\forall t\in [0,T],
\end{align}
which together with \eqref{eqEner02}, we can conclude the proof. 

\end{proof}

\begin{lemma}\label{PropLaplaciano}
Let $b=1$ and $M>0$ such that  $\alpha\in\mathcal{A}_{M}$. Let $(u,u_{t},u_{tt})$ be the unique solution of \eqref{eqmain2} with data $(u_0,u_1,u_2)\in (H^2(\Omega)\cap H_0^1(\Omega))\times H_0^1(\Omega)\times L^2(\Omega) $ and $f\in L^2(0,T; L^2(\Omega))$. Then, the term $\Delta u(t)$ can be bounded as follows
\begin{align*}
\|\Delta u(t)\|_{L^2(\Omega)}^2\leq  C \left(\|f\|^2_{L^2(0,T;L^2(\Omega))}+\overline{E}(0)+\|\Delta u_0\|_{L^2(\Omega)}^2\right),\;\forall t\in[0,T].
\end{align*}
\end{lemma}

\begin{proof}

Since the term $u_{tt}(t),\Delta u(t) \in L^2(\Omega)$, let us multiply the equation \eqref{1.6} by $L_0u$ and after integrating on $\Omega$, we deduce that
\begin{align}\label{1.61}
\D\frac{d}{dt}\|L_0u(t)\|_{L^2(\Omega)}^2+2c^2\|L_0u(t)\|_{L^2(\Omega)}^2=2\langle f(t)-\gamma u_{tt}(t),L_0u(t)\rangle_{L^2(\Omega)}.
\end{align}
%
%
By standard argument, from \eqref{1.61} we immediately obtain
\begin{align}\label{1.62}
\D\frac{d}{dt}\|L_0u(t)\|_{L^2(\Omega)}^2\leq \frac{1}{c^2}\|f(t)-\gamma u_{tt}(t)\|^2_{L^2(\Omega)}.
\end{align}

Integrating \eqref{1.62} from $0$ to $t>0$, we obtain that
\begin{align*}
\|L_0u(t)\|_{L^2(\Omega)}^2\Big|_{0}^{t}\leq \frac{1}{c^2}\int_0^{t}\|f(\tau)-\gamma u_{tt}(\tau)\|^2_{L^2(\Omega)} d\tau.
\end{align*}
Then, we have
\begin{multline*}
\|L_0u(t)\|_{L^2(\Omega)}^2-\|L_0u(t)\|_{L^2(\Omega)}^2\Big|_{t=0}\leq 
\frac{2}{c^2}\int_0^{t}\|f(\tau)\|^2_{L^2(\Omega)}d\tau+\\\frac{2}{c^2}\int_0^{t}\|\gamma u_{tt}(\tau)\|^2_{L^2(\Omega)}d\tau,
\end{multline*}
and then using Theorem \ref{Energy} we obtain the desired estimate.

\end{proof}

\subsection{Hidden regularity}

We can observe that the inverse problem considered in this paper needs that the normal derivative of the solution can be defined on the boundary. It is well known that the wave equation satisfies certain extra regularity called \emph{hidden regularity} \cite{lions1988controlabilite}. 
it is natural
to expect an analogous  result for  the  Moore--Gibson--Thompson equation, due its  hyperbolic nature  \cite{kaltenbacher2011wellposedness}.
In the following result, using the  multiplier method,  we obtain a hidden regularity for the solutions of this equation.

\begin{proposition}\label{Regularity}
The unique solution $(u,u_{t},u_{tt})\in C([0,T];(H^2(\Omega)\cap H_0^1(\Omega))\times H_0^1(\Omega)\times L^2(\Omega))$ of \eqref{eqmain2} satisfies 
\begin{align}\label{1.5}
\D\frac{\partial u}{\partial n}\in H^1(0,T;L^2(\Gamma)).
\end{align}
Moreover, the normal derivative satisfies
\begin{multline}\label{1.51}
\left\|\frac{\partial u}{\partial n}\right\|_{H^1(0,T;L^2(\Gamma))}^2\leq C\Big(\|u_0\|_{H^2(\Omega)\cap H_0^1(\Omega)}^2+\|u_1\|_{H_0^1(\Omega)}^2+\|u_2\|_{L^2(\Omega)}^2\\+\|f\|_{L^2(0,T;L^2(\Omega))}^2\Big).
\end{multline}
Consequently, the mapping 
\begin{align*}
(f,u_0,u_1,u_2)\mapsto \frac{\partial u}{\partial n}
\end{align*}
is linear continuous from $L^2(0,T;L^2(\Omega))\times (H^2(\Omega)\cap H_0^1(\Omega))\times H_0^1(\Omega)\times L^2(\Omega))$ into $H^1(0,T;L^2(\Gamma))$.
\end{proposition}

\begin{proof}
We use the multiplier method for the proof. Let $m\in W^{1,\infty}(\Omega;\R^{N})$ and let us multiply $L_0u$ by $m\nabla u$ and $L_0(u_t)$ by $m\nabla u_{t}$. Using the summation convention for repeated index, we obtain, respectively

\begin{multline}
\D
\int_{0}^{T}\int_{\Omega}L_0(u_{t})m\nabla u_{t} dxdt=\frac{1}{2}\int_0^{T}\int_{\Omega}\mbox{div}(m)| u_{tt}|^2dxdt +\int_{\Omega}u_{tt}m\nabla u_{t}\Big|_0^{T}dx\\
+\int_0^{T}\int_{\Omega}\frac{\partial u_{t}}{\partial x_{i}}\frac{\partial m_{j}}{\partial x_{i}}\frac{\partial u_{t}}{\partial x_{j}}dxdt
-\frac{1}{2}\int_0^{T}\int_{\Omega}\mbox{div}(m)|\nabla u_{t}|^2dxdr
-\frac{1}{2}\int_0^{T}\int_{\partial\Omega}|\nabla u_{t}\cdot n|^2(m\cdot n)d\sigma dt.
\end{multline}

and
\begin{multline}
\D
 \int_{0}^{T}\int_{\Omega}L_0um\nabla u dxdt=\frac{1}{2}\int_0^{T}\int_{\Omega}\mbox{div}(m)| u_{t}|^2dxdt+\int_{\Omega}u_{t}m\nabla u\Big|_0^{T}dx\\ 
+\int_0^{T}\int_{\Omega}\frac{\partial u}{\partial x_{i}}\frac{\partial m_{j}}{\partial x_{i}}\frac{\partial u}{\partial x_{j}}dxdt
-\frac{1}{2}\int_0^{T}\int_{\Omega}\mbox{div}(m)|\nabla u|^2dxdt
-\frac{1}{2}\int_0^{T}\int_{\partial\Omega}|\nabla u\cdot n|^2(m\cdot n)d\sigma dt.
\end{multline}

Now, taking the multiplier $m$ as a lifting of the outward unit normal $n$, so that $m\cdot n=1$, on $\Gamma$ and using that $(u,u_{t},u_{tt})\in C([0,T];(H^2(\Omega)\cap H_0^1(\Omega))\times H_0^1(\Omega)\times L^2(\Omega))$ we obtain

\begin{multline}\label{3-SZ}
\D 
\dfrac{1}{2} \int_0^{T}\int_{\partial\Omega}|\nabla u\cdot n|^2d\sigma dt+ \dfrac{1}{2} \int_0^{T}\int_{\partial\Omega}|\nabla u_{t}\cdot n|^2d\sigma dt\\
= -\int_{0}^{T}\int_{\Omega}(f-c^2L_0u-\gamma u_{tt})m\nabla u_{t} dxdt+\frac{1}{2}\int_0^{T}\int_{\Omega}\mbox{div}(m)| u_{tt}|^2dxdt \\
+\int_{\Omega}u_{tt}m\nabla u_{t}\Big|_0^{T}dx+\int_0^{T}\int_{\Omega}\frac{\partial u_{t}}{\partial x_{i}}\frac{\partial m_{j}}{\partial x_{i}}\frac{\partial u_{t}}{\partial x_{j}}dxdt
-\frac{1}{2}\int_0^{T}\int_{\Omega}\mbox{div}(m)|\nabla u_{t}|^2dxdr\\
 -\int_{0}^{T}\int_{\Omega}mL_0u\nabla u dxdt+\frac{1}{2}\int_0^{T}\int_{\Omega}\mbox{div}(m)| u_{t}|^2dxdt+\int_{\Omega}u_{t}m\nabla u\Big|_0^{T}dx\\ 
+\int_0^{T}\int_{\Omega}\frac{\partial u}{\partial x_{i}}\frac{\partial m_{j}}{\partial x_{i}}\frac{\partial u}{\partial x_{j}}dxdt
-\frac{1}{2}\int_0^{T}\int_{\Omega}\mbox{div}(m)|\nabla u|^2dxdt \\
\leq
C \left(\|f\|^2_{L^2(0,T;L^2(\Omega))}+\overline{E}(0)+\|\Delta u_0\|_{L^2(\Omega)}^2\right).
\end{multline}

 From \eqref{3-SZ}, using the continuous dependence of the solution with respect to the data, we obtain the desired estimate \eqref{1.51} and the proof is finished.

\end{proof}

\section{Proof of Main Results}\label{s3}

In this section we prove our main results, that is, Theorem \ref{mainresult} and Theorem \ref{observability}. First, we obtain the Carleman estimate given in Theorem \ref{observability} and then we apply this inequality to solve our inverse problem.

We use the following notation for the weighted energy of the wave operator $L_0$ 
\begin{align}\label{WEner}
W(y) := 
s\lambda \int_{0}^T\int_\Omega e^{2s\varphi_\lambda}\varphi_\lambda ( |y_t|^2 + |\nabla y|^2)dxdt 
+s^3\lambda^3 \int_{0}^T\int_\Omega e^{2s\varphi_\lambda}\varphi^3_\lambda |y|^2 dxdt,
\end{align}
with $\varphi_{\lambda}$ is given by \eqref{4-SZ}. Also, we recall  the operator $L$ defined in Section \ref{s2}:
\begin{align*}
Ly:=L_0 y_{t}+c^2L_0 y +\gamma y_{tt}.
\end{align*}

\begin{proof}[{\bf Proof of  Theorem \ref{observability}}]
Let $y\in L^2(0,T;H_0^1(\Omega))$ satisfying $Ly = f \in L^2(\Omega\times(0,T))$, $y(\cdot,0)=y_{t}(\cdot,0)=0$ in $\Omega$, and 
$y_{tt}(\cdot,0) = y_2 \in L^2(\Omega)$. 
Then, by \cite[Theorem 2.10]{kaltenbacher2012exponential} then
$(y,y_{t},y_{tt})\in C([0,T];(H^2(\Omega)\cap H_0^1(\Omega))\times H_0^1(\Omega)\times L^2(\Omega))$ 
and satisfies the boundary value problem
\begin{align}
\left\{
\begin{array}{ll}
L_0y_t+c^2L_0y+ \gamma y_{tt}= f ,  &\Omega \times (0,T) \\
y=0, & \Gamma \times (0,T). \\
y(\cdot,0) = 0, \medspace y_t(\cdot,0) =0 ,  \medspace y_{tt}(\cdot,0) = y_2,  & \Omega  
\end{array}
\right.
\end{align}

For a given function $F$ defined in $[0,T]$,  we will  denote   by $\widetilde F$ its  even extension, and by $\widehat F$ its odd extension  to $[-T,T]$.

Then $w = \widetilde{y}$ satisfies
\begin{align} \label{c21}
\left\{
\begin{array}{ll}
L_0w_t+\widehat{c^2}L_0w+\widehat{ \gamma} w_{tt}= \widehat{f} , & \Omega \times (-T,T) \\
w=0,  &\Gamma \times (-T,T). \\
w(\cdot,0) = 0, \medspace w_t(\cdot,0) =0 ,\medspace w_{tt}(\cdot,0) = y_2, &  \Omega.
\end{array}
\right.
\end{align}


%

We denote by $P$ the operator   
\begin{align*}
P:=\partial_{t}L_0+\widehat{c^2}L_0+\widehat{\gamma}\partial_{t}^2,
\end{align*}
and by $\|\cdot\|_{s}$ the  weighted norm
\begin{align*}
\|w\|_{s}^2:=\D\int_{-T}^{T}\int_{\Omega}e^{2s\varphi_{\lambda}}|w|^2dxdt,
\end{align*}
where $\varphi_{\lambda}$ is given by \eqref{4-SZ}. Then, 
\begin{align}\label{c4}
\|Pw- \widehat{\gamma}w_{tt}\|_{s}^2 
=
\|L_0w_{t}\|_{s}^2+c^4\|L_0w\|_{s}^2 +\D\int_{-T}^{T}\int_{\Omega}\widehat{c^2}e^{s\varphi_{\lambda}}\partial_{t}|L_0w|^2dxdt,
\end{align}
and, subsequently   
\begin{multline*}
\D\int_{-T}^{T}\int_{\Omega}\widehat{c^2}e^{s\varphi_{\lambda}}\partial_{t}|L_0w|^2dxdt=
\int_0^{T}\int_{\Omega}c^2e^{s\varphi_{\lambda}}\partial_{t}|L_0 w|^2dxdt  
-\int_{-T}^{0}\int_{\Omega}c^2e^{s\varphi_{\lambda}}\partial_{t}|L_0w|^2dxdt\\
\geq 
-2c^2\int_{\Omega}|L_0w(\cdot,0)|^2e^{s\varphi_{\lambda}(\cdot,0)}dx 
-\int_0^{T}\int_{\Omega}sc^2|L_0w|^2(\partial_{t}\varphi_{\lambda})e^{s\varphi_{\lambda}}dxdt 
\\+\int_{-T}^{0}\int_{\Omega}sc^2|L_0w|^2(\partial_{t}\varphi_{\lambda})e^{s\varphi_{\lambda}}dxdt.
\end{multline*}
Also, from the definition of the weight function, we  have 
\begin{align*}
\left\{
\begin{array}{ll}
\partial_{t}\varphi_{\lambda}<0, &\quad t\in(0,T),\\
\partial_{t}\varphi_{\lambda}>0, &\quad t\in(-T,0), 
\end{array}
\right.
\end{align*}
and then
\begin{align}\label{c2}
\D\int_{-T}^{T}\int_{\Omega}\widehat{c^2}e^{s\varphi_{\lambda}}\partial_{t}|L_0w|^2dxdt
\geq 
-2c^2\int_{\Omega}|L_0w(\cdot,0)|^2e^{s\varphi_{\lambda}(\cdot,0)}dx.
\end{align}

From \eqref{c4} and  \eqref{c2}, using that $w(\cdot,0)=0$, we deduce  that
\begin{align}\label{c5}
 \|L_0w_{t}\|_{s}^2+c^4\|L_0w\|_{s}^2-2c^2\int_{\Omega}|y_2(x)|^2e^{s\varphi_{\lambda}(\cdot,0)}dx
 \leq 
 \|Pw\|_{s}^2+ \|\widehat{\gamma}w_{tt}\|_{s}^2.
\end{align}

Hence, taking into account that $\phi(x,t) \leq \phi(x,0)$ for all $x \in \Omega$, and Lemma \eqref{Energy},
 we get 
 \begin{multline}
\int_{-T}^{T}\int_{\Omega}e^{2s\varphi_{\lambda}}|Pw|^2dxdt 
 \leq   
C \int_{0}^{T}\int_{\Omega}e^{2s\varphi_{\lambda}}|f|^2dxdt 
 + C \int_{\Omega}e^{2s\varphi_{\lambda}(\cdot,0) }|y_2|^2dxdt,
 \end{multline}
 which together with \eqref{c5} gives 
 \begin{align} 
 \|L_0w_{t}\|_{s}^2+c^4\|L_0w\|_{s}^2
 \leq  C  \int_{0}^{T}\int_{\Omega}e^{2s\varphi_{\lambda} }|f|^2dxdt 
  + C \int_{\Omega}e^{2s\varphi_{\lambda}(\cdot,0) }|y_2|^2dx
  + \|\widehat{\gamma}w_{tt}\|_{s}^2.  
  \label{acotPv}
 \end{align} 


Since 
$\widehat{\gamma} \in L^{\infty}(\Omega \times (-T,T))$,
from  \eqref{acotPv} we obtain that $L_0w$ and $L_0w_{t}$ belongs to $L^2(\Omega\times(-T,T))$. 
Therefore,  using the hidden regularity for the wave equation, we have that $\frac{\partial w}{\partial n}\in H^1(-T,T;L^2(\Gamma_0))$.  Then, we can apply the Carleman estimates given by Theorem 2.10 in \cite{baudouin2013global} for the wave equation to each term $L_0w$ and $L_0w_{t}$. Namely, we have
\begin{multline}\label{c6}
s\lambda\D\int_{-T}^{T}\int_{\Omega}e^{2s\varphi_{\lambda}}\varphi_{\lambda}(|w_{t}|^2+|\nabla w|^2)dxdt+s^3\lambda^3\D\int_{-T}^{T}\int_{\Omega}e^{2s\varphi_{\lambda}}\varphi_{\lambda}^3|w|^2 dxdt\\\leq C \int_{-T}^{T}\int_{\Omega}e^{2s\varphi_{\lambda}}|L_0 w|^2dxdt +Cs\lambda\int_{-T}^{T}\int_{\Gamma_0}e^{2s\varphi_{\lambda}}\left|\frac{\partial w}{\partial n}\right|^2d\sigma dt,
\end{multline}
where we use the fact that $w_{t}(\cdot,0)=0$, 
and
\begin{multline}\label{c7}
\sqrt{s}\int_{\Omega}e^{2s\varphi_{\lambda}(\cdot,0)}|y_2|^2dx+s\lambda\D\int_{-T}^{T}\int_{\Omega}e^{2s\varphi_{\lambda}}\varphi_{\lambda}(|w_{tt}|^2+|\nabla w_{t}|^2)dxdt\\+s^3\lambda^3\D\int_{-T}^{T}\int_{\Omega}e^{2s\varphi_{\lambda}}\varphi_{\lambda}^3|w_{t}|^2dxdt \leq C \int_{-T}^{T}\int_{\Omega}e^{2s\varphi_{\lambda}}|L_0 w_{t}|^2dxdt \\+Cs\lambda\int_{-T}^{T}\int_{\Gamma_0}e^{2s\varphi_{\lambda}}\left|\frac{\partial w_{t}}{\partial n}\right|^2d\sigma dt.
\end{multline}

Thus, from  \eqref{acotPv},  \eqref{c6} and \eqref{c7}  we obtain 
\begin{multline}\label{c9}
\sqrt{s}\int_{\Omega}e^{2s\varphi_{\lambda}(\cdot,0)}| y_2 |^2dx+c^4W(y)+W(y_{t}) 
\leq 
C  \int_{0}^{T}\int_{\Omega}e^{2s\varphi_{\lambda}(\cdot,0) }|f|^2dxdt
\\+ \|\widehat{\gamma}w_{tt}\|_{s}^2
+ C \int_{\Omega}|y_2|^2e^{s\varphi_{\lambda}(\cdot,0)}dx 
+s\lambda C \int_{-T}^T\int_{\Gamma_0} e^{2s\varphi_\lambda}\left(\left|\frac{\partial y_{t}}{\partial n}\right|^2 
+ c^4\left|\frac{\partial y}{\partial n}\right|^2\right)d\sigma dt.
\end{multline}

%

Then, there exists $s_0>0$ and $\lambda$ such that for every $s\geq s_0$ we absorb the second and third term in the right hand side of \eqref{c9} which implies
\begin{multline*}
\sqrt{s}\int_{\Omega}e^{2s\varphi_{\lambda}(\cdot,0)}|y_2|^2dx+c^4W(y)+W(y_{t})
\leq 
C  \int_{0}^{T}\int_{\Omega}e^{2s\varphi_{\lambda}(\cdot,0) }|f|^2dxdt
\\+s\lambda C \int_{0}^T\int_{\Gamma_0} e^{2s\varphi_\lambda}\left(\left|\frac{\partial y_{t}}{\partial n}\right|^2 
+ c^4\left|\frac{\partial y}{\partial n}\right|^2\right)d\sigma dt.
\end{multline*}

Finally, without loss of generality, we can take $M_0 > 0$ and $C > 1$ in definition \eqref{5-SZ} such that
$ \phi(x,0) \leq C \phi(x,t) $  for all $x \in \Omega$ and $t \in [0,T]$. Then we have
$ \varphi_{\lambda}(x,0) \leq C_1 \varphi_{\lambda}(x,t)$ for some $C_1 = C_1(\lambda)$ independent 
of $(x, t)  \in \Omega \times  [0,T]$, from where we conclude the desired estimate \eqref{1.133}.

\end{proof}

With the previous Carleman inequality, we can prove the main result of this article. 

\begin{proof}[{\bf Proof of  Theorem \ref{mainresult}}]
Using  the notation settled in
the previous section (see \eqref{gamma} and \eqref{1.6}), we write the MGT equation in the following way.
\begin{align}\label{NE}
\left\{
\begin{array}{ll}
L_0u_{t}+c^2L_0u+\gamma u_{tt}=f,  &\Omega\times(0,T) \\
u=h,  &\Gamma \times (0,T) \\
u(\cdot,0) = u_0, \medspace u_t(\cdot,0) = u_1, \medspace u_{tt}(\cdot,0) = u_2, &  \Omega.  
\end{array}
\right.
\end{align}
Hence, we will prove a stability estimate for 
coefficient $\gamma = \gamma(x)$ in  equation \eqref{NE}.

Let us denote by $u^k$ the  solution of  equation \eqref{NE} with coefficient $\gamma_k$, for $k=1, 2$, which 
existence is guaranteed by Theorem 2.10 in \cite{kaltenbacher2012exponential}.
Hence $z:= u^1 - u^2$ solves the following system.
\begin{align} \label{eqDif}
\left\{
\begin{array}{ll}
L_0z_t+c^2L_0z+\gamma_1(x)z_{tt}= (\gamma_2 - \gamma_1) R(x,t) ,  & \Omega \times (0,T) \\
z=0, & \Gamma \times (0,T) \\
z(\cdot,0) =  z_t(\cdot,0) =  z_{tt}(\cdot,0) =  0,  & \Omega  
\end{array}
\right.
\end{align}
where $R =  \partial_t^2 u^2$. 
%
Then  $y := \partial_t z$ satisfies
\begin{align} \label{eqDer}
\left\{
\begin{array}{ll}
L_0y_t+c^2L_0y+\gamma_1(x)y_{tt}= (\gamma_2 - \gamma_1) \partial_t R(x,t) , & \Omega \times (0,T) \\
y=0, & \Gamma \times (0,T) \\
y(\cdot,0) =  y_{t}(\cdot,0) = 0, \medspace  y_{tt}(\cdot,0) =  (\gamma_2 - \gamma_1)  R(x,0), &  \Omega  
\end{array}
\right.
\end{align}
Since $\gamma_2-\gamma_1$ belongs, in particular, to $L^2(\Omega)$ and $R\in H^1(0,T;L^{\infty}(\Omega))$, by Theorem 2.10 in \cite{kaltenbacher2012exponential}, we obtain that the Cauchy problem \eqref{eqDer} is well--posed and admits a unique solution
\begin{align*}
(y,y_{t},y_{tt})\in C([0,T]; (H^2(\Omega)\cap H_0^1(\Omega))\times H_0^1(\Omega)\times L^2(\Omega)).
\end{align*}
Moreover, from Theorem \ref{Regularity} the normal derivative $\frac{\partial y}{\partial n}$ belongs to $H^1(0,T;L^2(\Gamma))$ and satisfy
\begin{align*}
\left\|\frac{\partial y}{\partial n}\right\|_{H^1(0,T;L^2(\Gamma))}^2\leq C\|\gamma_2-\gamma_1\|_{L^2(\Omega)}^2(\|R(\cdot,0)\|_{L^{\infty}(\Omega)}^2+\|\partial_t R\|_{L^2(0,T;L^{\infty}(\Omega))}).
\end{align*}
This last estimate gives that $\frac{\partial z}{\partial n} \in H^2(0,T;L^2(\Gamma_0))$ and proves the
second inequality  of \eqref{stabi}.

Next, we apply Theorem \ref{observability} to $y$. From system \eqref{eqDer} we have
\begin{align*}
 \int_{0}^{T}\int_{\Omega}e^{2s\varphi_{\lambda} }|Ly|^2dxdt
 \leq C(\|\gamma_1\|_{L^{\infty}(\Omega)},\|\partial_{t}R\|_{L^2(0,T;L^{\infty}(\Omega))})\int_{\Omega}e^{2s\varphi_{\lambda}(\cdot,0)}|\gamma_2-\gamma_1|^2dx.
\end{align*}

Thus, from \eqref{1.133}
\begin{multline*}
\sqrt{s}\int_{\Omega}e^{2s\varphi_{\lambda}(\cdot,0)}|\gamma_2-\gamma_1|^2|R(x,0)|^2dx\leq
C\int_{\Omega}e^{2s\varphi_{\lambda}(\cdot,0)}|\gamma_2-\gamma_1|^2dx\\ +Cs\lambda\int_{0}^{T}\int_{\Gamma_0}e^{2s\varphi_{\lambda}}\left(\left|\frac{\partial y_{t}}{\partial n}\right|^2+c^4\left|\frac{\partial y}{\partial n}\right|^2\right)d\sigma dt,
\end{multline*}
which implies, using that
$ 
|R(x,0)| = |u_2| \geq \eta > 0 
$  a.e in $\Omega$,
\begin{multline*}
\eta^2\sqrt{s}\int_{\Omega}e^{2s\varphi_{\lambda}(\cdot,0)}|\gamma_2-\gamma_1|^2dx\leq
C\int_{\Omega}e^{2s\varphi_{\lambda}(\cdot,0)}|\gamma_2-\gamma_1|^2dx\\ +Cs\lambda\int_{0}^{T}\int_{\Gamma_0}e^{2s\varphi_{\lambda}}\left(\left|\frac{\partial y_{t}}{\partial n}\right|^2+c^4\left|\frac{\partial y}{\partial n}\right|^2\right)d\sigma dt.
\end{multline*}

Therefore, taking $s$ large enough we absorb the first term in the right hand side and we have
\begin{align*}
\eta^2\int_{\Omega}e^{2s\varphi_{\lambda}(\cdot,0)}|\gamma_2-\gamma_1|^2dx\leq
C\sqrt{s}\lambda\int_{0}^{T}\int_{\Gamma_0}e^{2s\varphi_{\lambda}}\left(\left|\frac{\partial y_{t}}{\partial n}\right|^2+c^4\left|\frac{\partial y}{\partial n}\right|^2\right)d\sigma dt,
\end{align*}
which is the  first  inequality  of \eqref{stabi} and the proof is finished.
\end{proof}

\section{Reconstruction of the coefficient}\label{s4}

In this section we shall propose an reconstruction algorithm for the  unknown parameter $\gamma$,  from measurements of the normal derivative of the solution $u(\gamma)$ of the MGT equation  \eqref{NE}. 
%
This algorithm is an extension of the work of Baudouin, Buhan and Ervedoza \cite{baudouin2013global}, in which they propose a reconstruction algorithm for the potential of the wave equation.

 By Theorem \ref{mainresult}, we known that the knowledge of $\frac{\partial u}{\partial n}$ on   $\Gamma_0\times (0,T)$ 
is enough to 
identify the parameter $\gamma$. Then $\alpha \in {\mathcal A}_M$ is equivalent to ask that $\gamma$ belongs to
\begin{align}\label{asumad}
\mathcal{B}_{M}:=\{\gamma \in L^{\infty}(\Omega),\quad 0\leq \gamma(x)\leq M, \quad \forall x\in\overline{\Omega}\}.
\end{align}

Let $\gamma\in \mathcal{B}_{M}$. Let $g\in L^2(\Omega\times(0,T))$ and $\mu\in H^1(0,T;L^2(\Gamma_0))$. 
Given $\varphi_\lambda$ defined in \eqref{4-SZ} with $\lambda >0$ given by Theorem \ref{observability}, we define the functional 
\begin{multline}\label{5.1}
J[\mu,g](y)=\frac{1}{2s}\int_{0}^{T}\int_{\Omega}e^{2s\varphi_{\lambda}}|Ly-g|^2dxdt\\+\frac{1}{2}\int_{0}^{T}\int_{\Gamma_0}e^{2s\varphi_{\lambda}}\left(\left|\frac{\partial y}{\partial n}-\mu\right|^2+\left|\frac{\partial y_{t}}{\partial n}-\mu_{t}\right|^2\right)d\sigma dt,
\end{multline}
defined in the space
\begin{align}\label{5.2}
\mathcal{V}=\{y\in L^2(0,T;H_0^1(\Omega))\mbox{ with }Ly\in L^2(\Omega\times(0,T)), y(\cdot,0)=y_{t}(\cdot,0)=0 \nonumber\\\mbox{ and } y_{tt}(\cdot,0)\in L^2(\Omega)\},
\end{align}
with the family of semi--norms
\begin{align}\label{5.3}
\|y\|_{\mathcal{V}, s}^2
:=
\frac{1}{s}\int_{0}^{T}\int_{\Omega}e^{2s\varphi_{\lambda}}|Ly|^2dxdt+\int_{0}^{T}\int_{\Gamma_0}e^{2s\varphi_{\lambda}}\left(\left|\frac{\partial y}{\partial n}\right|^2+\left|\frac{\partial y_{t}}{\partial n}\right|^2\right)d\sigma dt.
\end{align}

A few remarks about this semi--norms (for more details see \cite[Section 4]{baudouin2013global}):
\begin{remark}
\begin{enumerate}
\item\label{prue} Since the weighted functions $e^{s\varphi_{\lambda}}$ are bounded from below and from above by a positive constants depending on $s$, the semi--norms \eqref{5.3} are equivalent to
\begin{align*}
\|y\|_{\mathcal{V}}^2:=\int_{0}^{T}\int_{\Omega}|Ly|^2dxdt+\int_{0}^{T}\int_{\Gamma_0}\left(\left|\frac{\partial y}{\partial n}\right|^2+\left|\frac{\partial y_{t}}{\partial n}\right|^2\right)d\sigma dt,
\end{align*}
in the sense that there exists a constant $C=C(s)$, such that for all $y\in\mathcal{V}$
\begin{align*}
\frac{1}{C}\|y\|_{\mathcal{V}}^2\leq  
\|y\|_{\mathcal{V}, s}^2 \leq 
C\|y\|_{\mathcal{V}}^2.
\end{align*}

\item By Theorem \ref{observability}, there exists $s_0>0$ such that for every $s\geq s_0$ the semi--norm \eqref{5.3} is actually a norm. Hence, from 1. we have that  $\|\cdot\|_{\mathcal{V},s}$   is a norm for all $s>0$. 
In the rest of the paper, we will omit the subscript $s$ in the notation. 
 \end{enumerate}
\end{remark}

The first result concerning the reconstruction of $\gamma$, is to guarantee that the functional $J[\mu,g]$ reaches the minimum. Moreover, we have the following uniqueness result.

\begin{theorem}\label{minimo}
Assume the same hypotheses of Theorem \ref{observability} and assume that $g\in L^2(\Omega\times(0,T))$ and $\mu\in H^1(0,T;L^2(\Gamma_0))$. Then, for all $s>0$ and $\gamma\in \mathcal{B}_{M}$, the functional $J[\mu,g]$ defined by \eqref{5.1} is continuous, strictly convex and coercive on $\mathcal{V}$. Besides, the minimizer $y^{*}$ satisfies 
\begin{align*}
\|y^{*}\|_{\mathcal{V}}^2\leq\frac{4}{s}\int_{0}^{T}\int_{\Omega}e^{2s\varphi_{\lambda}}|g|^2dxdt+4\int_{0}^{T}\int_{\Gamma_0}e^{2s\varphi_{\lambda}}(|\mu|^2+|\mu_{t}|^2)d\sigma dt.
\end{align*}
\end{theorem}

\begin{proof}
The continuity and convexity is immediately. Let us see the coercivity.  
\begin{multline*}
J[\mu,g](y)= \frac{1}{2s}\int_{0}^{T}\int_{\Omega}e^{2s\varphi_{\lambda}}|Ly|^2dxdt+\frac{1}{2s}\int_{0}^{T}\int_{\Omega}e^{2s\varphi_{\lambda}}|g|^2dxdt
 -\frac{1}{s}\int_{0}^{T}\int_{\Omega}e^{2s\varphi_{\lambda}}gLydxdt\\+
\frac{1}{2}\int_{0}^{T}\int_{\Gamma_0}e^{2s\varphi_{\lambda}}\left(\left|\frac{\partial y}{\partial n}\right|^2
+\left|\frac{\partial y_{t}}{\partial n}\right|^2\right)d\sigma dt
 -\int_{0}^{T}\int_{\Gamma_0}e^{2s\varphi_{\lambda}}\left(\mu\frac{\partial y}{\partial n}+\mu_t\frac{\partial y_t}{\partial n}\right)\\+\frac{1}{2}\int_{0}^{T}\int_{\Gamma_0}e^{2s\varphi_{\lambda}}(|\mu|^2+|\mu_{t}|^2)d\sigma dt.
\end{multline*}
Using the fact that $2ab\leq 2a^2+\frac{b^2}{2}$, we deduce
\begin{multline*}
J[\mu,g](y)\geq \frac{1}{4s}\int_{0}^{T}\int_{\Omega}e^{2s\varphi_{\lambda}}|Ly|^2dxdt -\frac{1}{2s}\int_{0}^{T}\int_{\Omega}e^{2s\varphi_{\lambda}}|g|^2dxdt\\
\quad +\frac{1}{4}\int_{0}^{T}\int_{\Gamma_0}e^{2s\varphi_{\lambda}}\left(\left|\frac{\partial y}{\partial n}\right|^2
+\left|\frac{\partial y_{t}}{\partial n}\right|^2\right)d\sigma dt
-\frac{1}{2}\int_{0}^{T}\int_{\Gamma_0}e^{2s\varphi_{\lambda}}(|\mu|^2+|\mu_{t}|^2)d\sigma dt\\
=\frac{1}{4}\|y\|_{\mathcal{V}}^2-\frac{1}{s}\int_{0}^{T}\int_{\Omega}e^{2s\varphi_{\lambda}}|g|^2dxdt-\int_{0}^{T}\int_{\Gamma_0}e^{2s\varphi_{\lambda}}(|\mu|^2+|\mu_{t}|^2)d\sigma dt.
\end{multline*}

Therefore, the functional $J[\mu,g]$ admits a unique minimizer $y^{*}$ in $\mathcal{V}$.

Now, let us prove the estimates on the minimizer. First, we develop the functional $J[\mu,g](y^{*})$:
\begin{multline*}
J[\mu,g](y^{*})=\frac{1}{2s}\int_{0}^{T}\int_{\Omega}e^{2s\varphi_{\lambda}}|Ly^{*}|^2dxdt+\frac{1}{2}\int_{0}^{T}\int_{\Gamma_0}e^{2s\varphi_{\lambda}}\left(\left|\frac{\partial y^{*}}{\partial n}\right|^2+\left|\frac{\partial y_{t}^{*}}{\partial n}\right|^2\right)d\sigma dt\\
\quad +\frac{1}{2s}\int_{0}^{T}\int_{\Omega}e^{2s\varphi_{\lambda}}|g|^2dxdt+\frac{1}{2}\int_{0}^{T}\int_{\Gamma_0}e^{2s\varphi_{\lambda}}(|\mu|^2+|\mu_{t}|^2)d\sigma dt\\
\quad -\frac{1}{s}\int_{0}^{T}\int_{\Omega}e^{2s\varphi_{\lambda}}gLy^{*} dxdt-\int_{0}^{T}\int_{\Gamma_0}e^{2s\varphi_{\lambda}}\left(\mu\frac{\partial y^{*}}{\partial n}+\mu_{t}\frac{\partial y_{t}^{*}}{\partial n}\right)d\sigma dt.
\end{multline*}

Next, since $y^{*}$ is the minimizer, we have that $J[\mu,g](y^{*})\leq J[\mu,g](0)$, which implies in particular 
\begin{multline*}
\frac{1}{2s}\int_{0}^{T}\int_{\Omega}e^{2s\varphi_{\lambda}}|Ly^{*}|^2dxdt+\frac{1}{2}\int_{0}^{T}\int_{\Gamma_0}e^{2s\varphi_{\lambda}}\left(\left|\frac{\partial y^{*}}{\partial n}\right|^2+\left|\frac{\partial y_{t}^{*}}{\partial n}\right|^2\right)d\sigma dt\\\leq \frac{1}{s}\int_{0}^{T}\int_{\Omega}e^{2s\varphi_{\lambda}}gLy^{*}dxdt+\int_{0}^{T}\int_{\Gamma_0}e^{2s\varphi_{\lambda}}\left(\mu\frac{\partial y^{*}}{\partial n}+\mu_{t}\frac{\partial y_{t}^{*}}{\partial n}\right)d\sigma dt
\end{multline*}
Therefore, using that $2ab\leq 2a^2+\frac{b^2}{2}$ and the definition of the norm $\|\cdot\|_{\mathcal{V}}$, we deduce
\begin{align*}
\frac{1}{4}\|y^{*}\|_{\mathcal{V}}^2\leq \frac{1}{s}\int_{0}^{T}\int_{\Omega}e^{2s\varphi_{\lambda}}|g|^2dxdt+\int_{0}^{T}\int_{\Gamma_0}e^{2s\varphi_{\lambda}}(|\mu|^2+|\mu_{t}|^2)d\sigma dt.
\end{align*}

\end{proof}

Secondly, the following Theorem gives a relationship between the unique minimizer of $J[\mu,g]$ and $g$. This is, together with the Theorem \ref{observability}, an essential result for the  proof of convergence of our algorithm of reconstruction. 

\begin{theorem}\label{minimo2}
Assume the same hypotheses of Theorem \ref{observability} and assume that $\mu\in H^1(0,T;L^2(\Gamma_0))$ and $g^1,g^2\in L^2(\Omega\times(0,T))$. Let $y^{*,i}$ be the unique minimizer of the functional $J[\mu,g^{i}]$, for $i=1,2$. Then, there exists $s_0>0$ and a constant $C>0$ such that for all $s\geq s_0$
\begin{align}\label{5.4}
\sqrt{s}\int_{\Omega}e^{2s\varphi_{\lambda}(\cdot,0)}|y_{tt}^{*,1}(\cdot,0)-y_{tt}^{*,2}(\cdot,0)|
^2dx\leq C\int_{0}^{T}\int_{\Omega}e^{2s\varphi_{\lambda}}|g^1-g^2|^2dxdt.
\end{align}
\end{theorem}

\begin{proof}
Since $y^{*,i}$ is the unique minimizer of $J[\mu,g^{i}]$, for $i=1,2$, we have that for all $y\in\mathcal{V}$
\begin{multline}\label{5.5}
\frac{1}{s}\int_{0}^{T}\int_{\Omega}e^{2s\varphi_{\lambda}}(Ly^{*,1}-g^1)Lydxdt\\+\int_{0}^{T}\int_{\Gamma_0}e^{2s\varphi_{\lambda}}\left[\left(\frac{\partial y^{*,1}}{\partial  n}-\mu\right)\frac{\partial y}{\partial n}+\left(\frac{\partial y_{t}^{*,1}}{\partial  n}-\mu_{t}\right)\frac{\partial y_{t}}{\partial n}\right]d\sigma dt=0,
\end{multline}
and
\begin{multline}\label{5.6}
\frac{1}{s}\int_{0}^{T}\int_{\Omega}e^{2s\varphi_{\lambda}}(Ly^{*,2}-g^2)Lydxdt\\+\int_{0}^{T}\int_{\Gamma_0}e^{2s\varphi_{\lambda}}\left[\left(\frac{\partial y^{*,2}}{\partial  n}-\mu\right)\frac{\partial y}{\partial n}+\left(\frac{\partial y_{t}^{*,2}}{\partial  n}-\mu_{t}\right)\frac{\partial y_{t}}{\partial n}\right]d\sigma dt=0.
\end{multline}
Subtracting \eqref{5.5} and \eqref{5.6}, for $y=y^{*,1}-y^{*,2}$, we deduce that
\begin{multline*}
\frac{1}{s}\int_{0}^{T}\int_{\Omega}e^{2s\varphi_{\lambda}}|Ly|^2dxdt+\int_{0}^{T}\int_{\Gamma_0}e^{2s\varphi_{\lambda}}\left(\left|\frac{\partial y}{\partial n}\right|^2+\left|\frac{\partial y_{t}}{\partial n}\right|^2\right)d\sigma dt\\=\frac{1}{s}\int_{0}^{T}\int_{\Omega}e^{2s\varphi_{\lambda}}(g^1-g^2)Ly dxdt,
\end{multline*}
Then, applying again $2ab\leq 2a^2+\frac{b^2}{2}$ we obtain 
\begin{multline}\label{5.7}
\frac{1}{2}\int_{0}^{T}\int_{\Omega}e^{2s\varphi_{\lambda}}|Ly|^2dxdt+s\int_{0}^{T}\int_{\Gamma_0}e^{2s\varphi_{\lambda}}\left(\left|\frac{\partial y}{\partial n}\right|^2+\left|\frac{\partial y_{t}}{\partial n}\right|^2\right)d\sigma dt\\\leq2\int_{0}^{T}\int_{\Omega}e^{2s\varphi_{\lambda}}|g^1-g^2|^2 dxdt,
\end{multline}

Finally, by the estimate \eqref{1.133} of Theorem \ref{observability} we obtain the desired result. 

\end{proof}

Finally, we present our algorithm and the convergence result of this.\\

\hrulefill

{\bf Algorithm:}

\hrulefill
\begin{enumerate}
\item {\bf Initialization:} $\gamma^0=0$.
\item {\bf Iteration:} From $k$ to $k+1$
\begin{enumerate}
\item[Step 1] - Given $\gamma^{k}$ we consider  $\mu^{k}=\partial_{t}\left(\frac{\partial u(\gamma^{k})}{\partial n}-\frac{\partial u(\gamma)}{\partial n}\right)$ and  \\$\mu_{t}^{k}=\partial_{t}\left(\frac{\partial u_{t}(\gamma^{k})}{\partial n}-\frac{\partial u_{t}(\gamma)}{\partial n}\right)$ on $\Gamma_0\times(0,T)$, where $u(\gamma^{k})$ and $u(\gamma)$ are the solution of the problems
\begin{align}
\left\{
\begin{array}{ll}
L_0u_t+c^2L_0u+\gamma^k(x)u_{tt}=f,  & \Omega \times (0,T) \\
u=g, & \Gamma \times (0,T) \\
u(\cdot,0) = u_0, \medspace u_t(\cdot,0) = u_1, \medspace u_{tt}(\cdot,0) = u_2,  & \Omega  
\end{array}
\right.
\end{align}
and
\begin{align}
\left\{
\begin{array}{ll}
L_0u_t+c^2L_0u+\gamma(x)u_{tt}=f, & \Omega \times (0,T) \\
u=g, & \Gamma \times (0,T) \\
u(\cdot,0) = u_0, \medspace u_t(\cdot,0) = u_1, \medspace u_{tt}(\cdot,0) = u_2, & \Omega,
\end{array}
\right.
\end{align}
respectively. 

\item[Step 2] - Minimize the functional $J[\mu^{k},0]$ on the admissible  trajectories $y\in\mathcal{V}$:
\begin{multline*}
J[\mu^{k},0](y)=\frac{1}{2s}\int_{0}^{T}\int_{\Omega}e^{2s\varphi_{\lambda}}|L_0y_t+c^2L_0y+\gamma^k(x)y_{tt}|^2dxdt\\+\frac{1}{2}\int_{0}^{T}\int_{\Gamma_0}e^{2s\varphi_{\lambda}}\left(\left|\frac{\partial y}{\partial n}-\mu^{k}\right|^2+\left|\frac{\partial y_{t}}{\partial n}-\mu_{t}^{k}\right|^2\right)d\sigma dt
\end{multline*}

\item[Step 3] - Let $y^{*,k}$ the minimizer of $J[\mu^{k},0]$ and 
\begin{align}\label{5.10}
\widetilde{\gamma}^{k+1}=\gamma^{k}+\frac{y_{tt}^{*,k}(\cdot,0)}{u_2}.
\end{align}

\item[Step 4] - Finally, consider $\gamma^{k+1}=T(\widetilde{\gamma}^{k+1})$, where 
\begin{align}\label{5.11}
T(\gamma)=
\left\{
\begin{array}{ll}
M\quad &\mbox{ if }\gamma(x)>M\\
\\
\gamma\quad &\mbox{ if } 0\leq \gamma(x)\leq M\\
\\
0\quad &\mbox{ if }\gamma(x)<0.
\end{array}
\right.
\end{align}
\end{enumerate}
\end{enumerate}
\hrulefill\\

Therefore, under the previous Theorem \ref{minimo2}, we can prove the convergence of this algorithm:

\begin{theorem}\label{convergencia}
Assume the same hypotheses of Theorem \ref{observability}, and the following assumption of $u(\gamma):$
\begin{align}
u(\gamma)\in H^3(0,T;L^{\infty}(\Omega)) \mbox{ and } |u_2|\geq\eta>0.
\end{align}

Then, there exists a constant $C>0$ and $s_0>0$ such that for all $s\geq s_0$ and $k\in\N$
\begin{align}\label{4.13}
\int_{\Omega}e^{2s\varphi_{\lambda}(\cdot,0)}(\gamma^{k+1}-\gamma)^2dx\leq \frac{C}{\sqrt{s}}\int_{\Omega}e^{2s\varphi_{\lambda}(\cdot,0)}(\gamma^{k}-\gamma)^2dx.
\end{align}
\end{theorem}

\begin{proof}
We consider $y^{k}=\partial_{t}(u(\gamma^{k})-u(\gamma))$, which is the solution of
\begin{align}\label{5.12}
\left\{
\begin{array}{ll}
L_0y_t^{k}+c^2L_0y^{k}+\gamma(x)^{k}y_{tt}^{k}=(\gamma-\gamma^{k})\partial_{t}R(x,t), & \Omega \times (0,T) \\
y^{k}=0,  & \Gamma \times (0,T) \\
y^{k}(\cdot,0) =0 ,\medspace  y_t^k(\cdot,0) = 0, \medspace y_{tt}^{k}(\cdot,0) = (\gamma-\gamma^{k})R(x,0),  & \Omega  
\end{array}
\right.
\end{align}
where $R(x,t)=\partial_{t}^2u(\gamma)$. Thus,
\begin{align}\label{mu}
\mu^{k}=\frac{\partial y^{k}}{\partial  n},\qquad \mu_{t}^{k}=\frac{\partial y_{t}^{k}}{\partial n}.
\end{align}

We observe that $y^{k}$ belongs to $\mathcal{V}$. Therefore, by \eqref{mu}, the solution $y^{k}$ of \eqref{5.12} satisfy the Euler--Lagrange equations associated to the functional $J[\mu^{k},g^{k}]$, where $g^{k}=(\gamma-\gamma^{k})\partial_{t}R(x,t)$. Since $J[\mu^{k},g^{k}]$ admits a unique minimizer, $y^{k}$ corresponds to minimum of $J[\mu^{k},g^{k}]$.

Let $y^{*,k}$ be the minimizer of $J[\mu^{k},0].$ From Theorem \ref{minimo2} we obtain that
\begin{align}
\sqrt{s}\int_{\Omega}e^{2s\varphi_{\lambda}(\cdot,0)}|y_{tt}^{*,k}(\cdot,0)-y_{tt}^{k}(\cdot,0)|
^2dx\leq C\int_{0}^{T}\int_{\Omega}e^{2s\varphi_{\lambda}}|(\gamma-\gamma^{k})\partial_{t}R(x,t)|^2dxdt.
\end{align}

From \eqref{5.10} and \eqref{5.12}
\begin{align*}
y_{tt}^{*,k}(\cdot,0)=(\widetilde{\gamma}^{k+1}-\gamma^{k})u_2, \qquad y_{tt}^{k}(\cdot,0)=(\gamma-\gamma^{k})u_2.
\end{align*}
This implies that, using that $|u_2|\geq\eta>0$
\begin{align}
\eta^2\sqrt{s}\int_{\Omega}e^{2s\varphi_{\lambda}(\cdot,0)}(\widetilde{\gamma}^{k+1}-\gamma)^2
dx\leq C\int_{0}^{T}\int_{\Omega}e^{2s\varphi_{\lambda}}|(\gamma-\gamma^{k})\partial_{t}R(x,t)|^2dxdt.
\end{align}

Since  the function $T$ defined in \eqref{5.11} is Lipschitz continuous and satisfy $T(\gamma)=\gamma$, we obtain
\begin{align}
|\widetilde{\gamma}^{k+1}-\gamma|\geq|T(\widetilde{\gamma}^{k+1})-T(\gamma)|=|\gamma^{k+1}-\gamma|.
\end{align}

On the other hand, since $\phi(\cdot,t)$ is decreasing in $t\in(0,T)$ and $\partial_{t}R(x,t)\in L^2(0,T;L^{\infty}(\Omega))$, we conclude 
\begin{align*}
\int_{\Omega}e^{2s\varphi_{\lambda}(\cdot,0)}(\gamma^{k+1}-\gamma)^2
dx\leq \frac{C}{\sqrt{s}}\frac{\|\partial_{t}R(x,t)\|_{L^2(0,T;L^{\infty}(\Omega)}^{}}{\eta^2}\int_{\Omega}e^{2s\varphi_{\lambda}(\cdot,0)}(\gamma-\gamma^{k})^2dxdt.
\end{align*}

\end{proof}

Let us finish this section with the following observation.
\begin{remark}
We can observe that this algorithm, from a theoretical point of view, is based on the minimization of a convex and coercive functional. Therefore, we can expect that numerical simulations can be done using, for instance, CasADi open-source tool for nonlinear optimization and algorithmic differentiation \cite{Andersson2019}. However, some drawbacks appears in its numerical simulations. This can be seen in the definition of the functional $J[\mu,g]$, which involves two exponentials
\begin{align*}
e^{2s\varphi_\lambda}=e^{2se^{\lambda\phi}}.
\end{align*}
Both parameters $\lambda$ and $s$ are chosen large enough, in order to use the Carleman estimate given in Theorem \ref{observability}. This implies a immediately problem from a numerical point of view. For example, if we consider $s=\lambda=3$, $\Omega=(0,1)$, $x_0=0$, $T=1$ and $\beta=1$, the following ratio
\begin{align*}
\frac{\max_{\Omega\times (0,T)}e^{2s\varphi_\lambda}}{\min_{\Omega\times (0,T)}e^{2s\varphi_\lambda}}
\end{align*}
is of order of $10^{340}$ (see for instance \cite{MR3670259}). 

It seems reasonable  modify the algorithm presented here in order to obtain a numerical implementation, to validate at least with an example, the coefficient inverse problem studied in this article. In this direction, the modified algorithm  is part of our forthcoming work.
\end{remark}

\bibliographystyle{abbrv}
\bibliography{References1}

\begin{thebibliography}{10}

\bibitem{Andersson2019}
J.~A.~E. Andersson, J.~Gillis, G.~Horn, J.~Rawlings, and M.~Diehl.
\newblock {CasADi} -- {A} software framework for nonlinear optimization and
  optimal control.
\newblock {\em Math. Program. Comput.}, 11(1):1--36, 2019.

\bibitem{bardos1992sharp}
C.~Bardos, G.~Lebeau, and J.~Rauch.
\newblock Sharp sufficient conditions for the observation, control, and
  stabilization of waves from the boundary.
\newblock {\em SIAM J. Control Optim.}, 30(5):1024--1065, 1992.

\bibitem{baudouin2010lipschitz}
L.~Baudouin.
\newblock Lipschitz stability in an inverse problem for the wave equation.
  {Preprint}, available at: http://hail.archives-ouvertes.fr/hal-00598876/fr/.
\newblock 2010.

\bibitem{baudouin2013global}
L.~Baudouin, M.~De~Buhan, and S.~Ervedoza.
\newblock Global {C}arleman estimates for waves and applications.
\newblock {\em Comm. Partial Differential Equations}, 38(5):823--859, 2013.

\bibitem{MR3670259}
L.~Baudouin, M.~de~Buhan, and S.~Ervedoza.
\newblock Convergent algorithm based on {C}arleman estimates for the recovery
  of a potential in the wave equation.
\newblock {\em SIAM J. Numer. Anal.}, 55(4):1578--1613, 2017.

\bibitem{beilina2012approximate}
L.~Beilina and M.~V. Klibanov.
\newblock {\em Approximate global convergence and adaptivity for coefficient
  inverse problems}.
\newblock Springer Science \& Business Media, 2012.

\bibitem{beilina2015globally}
L.~Beilina and M.~V. Klibanov.
\newblock Globally strongly convex cost functional for a coefficient inverse
  problem.
\newblock {\em Nonlinear Anal. Real World Appl.}, 22:272--288, 2015.

\bibitem{Bella04}
M.~Bellassoued.
\newblock Uniqueness and stability in determining the speed of propagation of
  second-order hyperbolic equation with variable coefficients.
\newblock {\em Appl. Anal.}, 83(10):983--1014, 2004.

\bibitem{BellaYama}
M.~Bellassoued and M.~Yamamoto.
\newblock {\em Carleman estimates and applications to inverse problems for
  hyperbolic systems}.
\newblock Springer Monographs in Mathematics. Springer, Tokyo, 2017.

\bibitem{BK}
A.~L. Bukhge\u{\i}m and M.~V. Klibanov.
\newblock Uniqueness in the large of a class of multidimensional inverse
  problems.
\newblock {\em Dokl. Akad. Nauk SSSR}, 260(2):269--272, 1981.

\bibitem{conejero2015chaotic}
J.~A. Conejero, C.~Lizama, and F.~Rodenas.
\newblock Chaotic behaviour of the solutions of the {M}oore-{G}ibson-{T}hompson
  equation.
\newblock {\em Appl. Math. Inf. Sci.}, 9(5):2233--2238, 2015.

\bibitem{fursikov1996controllability}
A.~V. Fursikov and O.~Y. Imanuvilov.
\newblock {\em Controllability of evolution equations}, volume~34 of {\em
  Lecture Notes Series}.
\newblock Seoul National University, Research Institute of Mathematics, Global
  Analysis Research Center, Seoul, 1996.

\bibitem{ho}
L.~F. Ho.
\newblock Observabilit\'{e} fronti\`ere de l'\'{e}quation des ondes.
\newblock {\em C. R. Acad. Sci. Paris S\'{e}r. I Math.}, 302(12):443--446,
  1986.

\bibitem{MR1964256}
O.~Y. Imanuvilov and M.~Yamamoto.
\newblock Determination of a coefficient in an acoustic equation with a single
  measurement.
\newblock {\em Inverse Problems}, 19(1):157--171, 2003.

\bibitem{kaltenbacher2015mathematics}
B.~Kaltenbacher.
\newblock Mathematics of nonlinear acoustics.
\newblock {\em Evol. Equ. Control Theory}, 4(4):447--491, 2015.

\bibitem{kaltenbacher2012exponential}
B.~Kaltenbacher and I.~Lasiecka.
\newblock Exponential decay for low and higher energies in the third order
  linear {M}oore-{G}ibson-{T}hompson equation with variable viscosity.
\newblock {\em Palest. J. Math.}, 1(1):1--10, 2012.

\bibitem{kaltenbacher2011wellposedness}
B.~Kaltenbacher, I.~Lasiecka, and R.~Marchand.
\newblock Wellposedness and exponential decay rates for the
  {M}oore-{G}ibson-{T}hompson equation arising in high intensity ultrasound.
\newblock {\em Control Cybernet.}, 40(4):971--988, 2011.

\bibitem{kaltenbacher2012well}
B.~Kaltenbacher, I.~Lasiecka, and M.~K. Pospieszalska.
\newblock Well-posedness and exponential decay of the energy in the nonlinear
  {J}ordan-{M}oore-{G}ibson-{T}hompson equation arising in high intensity
  ultrasound.
\newblock {\em Math. Models Methods Appl. Sci.}, 22(11):1250035, 34, 2012.

\bibitem{lions1988controlabilite}
J.~L. Lions.
\newblock {\em Contr{\^o}labilit{\'e} exacte perturbations et stabilisation de
  syst{\`e}mes distribu{\'e}s. Tome 1, Contr{\^o}labilit{\'e} exacte.},
  volume~8.
\newblock Recherches en Mathematiques Appliqu{\'e}es, Masson, 1988.

\bibitem{liu2011global}
S.~Liu and R.~Triggiani.
\newblock Global uniqueness and stability in determining the damping and
  potential coefficients of an inverse hyperbolic problem.
\newblock {\em Nonlinear Anal. Real World Appl.}, 12(3):1562--1590, 2011.

\bibitem{liu2014inverse}
S.~Liu and R.~Triggiani.
\newblock Inverse problem for a linearized {J}ordan-{M}oore-{G}ibson-{T}hompson
  equation.
\newblock In {\em New prospects in direct, inverse and control problems for
  evolution equations}, volume~10 of {\em Springer INdAM Ser.}, pages 305--351.
  Springer, Cham, 2014.

\bibitem{lizama2019controllability}
C.~Lizama and S.~Zamorano.
\newblock Controllability results for the {M}oore-{G}ibson-{T}hompson equation
  arising in nonlinear acoustics.
\newblock {\em J. Differential Equations}, 266(12):7813--7843, 2019.

\bibitem{marchand2012abstract}
R.~Marchand, T.~McDevitt, and R.~Triggiani.
\newblock An abstract semigroup approach to the third-order
  {M}oore-{G}ibson-{T}hompson partial differential equation arising in
  high-intensity ultrasound: structural decomposition, spectral analysis,
  exponential stability.
\newblock {\em Math. Methods Appl. Sci.}, 35(15):1896--1929, 2012.

\bibitem{pellicer2019optimal}
M.~Pellicer and J.~Sol\`a-Morales.
\newblock Optimal scalar products in the {M}oore-{G}ibson-{T}hompson equation.
\newblock {\em Evol. Equ. Control Theory}, 8(1):203--220, 2019.

\bibitem{PuelYama}
J.-P. Puel and M.~Yamamoto.
\newblock On a global estimate in a linear inverse hyperbolic problem.
\newblock {\em Inverse Problems}, 12(6):995--1002, 1996.

\bibitem{yamamoto1999uniqueness}
M.~Yamamoto.
\newblock Uniqueness and stability in multidimensional hyperbolic inverse
  problems.
\newblock {\em J. Math. Pures Appl. (9)}, 78(1):65--98, 1999.

\bibitem{Yama99}
M.~Yamamoto.
\newblock Uniqueness and stability in multidimensional hyperbolic inverse
  problems.
\newblock {\em J. Math. Pures Appl. (9)}, 78(1):65--98, 1999.

\bibitem{yamamoto2003one}
M.~Yamamoto.
\newblock One unique continuation for a linearized {B}enjamin-{B}ona-{M}ahony
  equation.
\newblock {\em J. Inverse Ill-Posed Probl.}, 11(5):537--543, 2003.

\bibitem{MR3774702}
J.~Yu, Y.~Liu, and M.~Yamamoto.
\newblock Theoretical stability in coefficient inverse problems for general
  hyperbolic equations with numerical reconstruction.
\newblock {\em Inverse Problems}, 34(4):045001, 30, 2018.

\bibitem{zhang2000explicit}
X.~Zhang.
\newblock Explicit observability inequalities for the wave equation with lower
  order terms by means of {C}arleman inequalities.
\newblock {\em SIAM J. Control Optim.}, 39(3):812--834, 2000.

\end{thebibliography}

\end{document}